\documentclass{amsart}
\usepackage{amsmath}
\usepackage{amssymb}
\usepackage{amscd}
\usepackage{color}

\newcommand{\e}{\frak{e}}
\newcommand{\Max}{\operatorname{Max}}
\newcommand{\Ad}{\operatorname{Ad}}
\newtheorem{theorem}{Theorem}[section]
\newtheorem{lemma}[theorem]{Lemma}
\newtheorem{corollary}[theorem]{Corollary}
\newtheorem{proposition}[theorem]{Proposition}
\newtheorem{definition}[theorem]{Definition}

\newtheorem{remark}[theorem]{Remark}

\newcommand{\Star}{\operatorname{Star}}
\newcommand{\HStar}{\operatorname{HStar}}
\newcommand{\SStar}{\operatorname{SStar}}

\begin{document}

\title[Homogeneous preserving star operations]
{On a particular subspace of homogeneous preserving star operations}
\author[P. Sahandi]
{Parviz Sahandi}

\address{Department of Pure Mathematics, Faculty of Mathematical sciences, University of Tabriz, Tabriz,
51666-15648, Iran.} \email{sahandi@ipm.ir}

%\date{\today}

\thanks{2020 Mathematics Subject Classification: 13A15, 13G05, 13A02, 13F05}
\thanks{Key Words and Phrases: graded integral domain, homogeneous preserving star operation, star operation, spectral space}

\begin{abstract}
Let $\Gamma$ be a torsionless commutative cancellative monoid,
$R=\bigoplus_{\alpha \in \Gamma}R_{\alpha}$ be a $\Gamma$-graded integral
domain. In this note we show that each homogeneous star operation $\star:\mathbf{HF}(R)\to\mathbf{HF}(R)$ of $R$, is the restriction of a (classical) star operation $\e(\star):\mathbf{F}(R)\to\mathbf{F}(R)$ of $R$, that is $\e(\star)|_{\mathbf{HF}(R)}=\star$. We also show that the set $\HStar_f(R)$ of homogeneous star operations of finite type on $R$, endowed with the Zariski topology, is a spectral space.
\end{abstract}

\maketitle

\section{Introduction}

The star operations are defined by axioms selected by Krull among the
properties satisfied by some classical operations, such as the $v$-operation,
the $t$-operation and the completion. Star operations have shown to be an essential tool in Multiplicative Ideal
Theory, allowing a new approach for characterizing several classes of integral
domains. For example, an integrally closed domain $D$ is a Pr\"{u}fer domain if
and only if $I^t = I$ for each nonzero ideal $I$ of $D$, \cite[Proposition 34.12]{g72}.

Let $R$ be an integral domain with quotient field $K$. Let $\mathbf{F}(R)$ be the set of all nonzero fractional ideals of $R$. As in \cite[Section 32]{g72}, a {\it (classical) star operation} on $R$ is a map $\star:\mathbf{F}(R)\rightarrow\mathbf{F}(R)$, $A\mapsto A^{\star}$, such that, for all $0\neq x\in K$, and for all $A, B\in\mathbf{F}(R)$, the following three properties hold:
\begin{description}
  \item[$\star_1$] $(x)^{\star}=(x)$ and $(xA)^{\star}=xA^{\star}$;
  \item[$\star_2$] $A\subseteq A^{\star}$ and $A\subseteq B \Rightarrow A^{\star}\subseteq B^{\star}$;
  \item[$\star_3$] $A^{\star\star}:=(A^{\star})^{\star}=A^{\star}$.
\end{description}
Let $\Gamma$ be a torsion-free cancellative monoid, and $R=\bigoplus_{\alpha \in \Gamma}R_{\alpha}$ be a graded integral domain. In \cite{s14,s22}, the author considered (classical) star operations (and more generally semistar operations) $\star$ on $R$ such that $\star$ sends homogeneous fractional ideals to homogeneous ones. Such a star operation is called \emph{homogenous preserving star operation}. The identity $d$, $v$, $t$, $w$ and the completion $b$ are various examples of homogenous preserving star operations (see \cite[Example 1]{s22}). The set of (classical) star operations (resp. homogeneous preserving star operations) on $R$ is denoted by $\Star(R)$ (resp. $\Star_{hp}(R)$).

Recently the notion of homogeneous star operations on graded integral domains $R=\bigoplus_{\alpha \in \Gamma}R_{\alpha}$ was introduced in \cite{k24} (see also \cite[Definition 3.4]{FO12}). Let $\mathbf{HF}(R)$ be the set of nonzero homogeneous fractional ideals of $R$. Obviously $\mathbf{HF}(R)\subseteq \mathbf{F}(R)$. Let $H$ be the set of nonzero homogeneous elements of $R$, and $R_H$ the homogeneous quotient field of $R$. A \emph{homogeneous star operation} on $R$ is a map $*:\mathbf{HF}(R)\to\mathbf{HF}(R)$, $A\mapsto A^*$ satisfying the above three conditions $\star_1, \star_2$ and $\star_3$, for all homogeneous element $x\in R_H\setminus\{0\}$ and $A, B\in \mathbf{HF}(R)$. The set of homogeneous star operations on $R$ is denoted by $\HStar(R)$.

Clearly, if $\star:\mathbf{F}(R)\rightarrow\mathbf{F}(R)$ is a homogenous preserving star operation, then the restricted mapping $\star|_{\mathbf{HF}(R)}:\mathbf{HF}(R)\to\mathbf{HF}(R)$ is a homogeneous star operation. %(we still show $\star|_{\mathbf{HF}(R)}$ by $\star$).

\begin{definition}\label{d}
Let $\star:\mathbf{HF}(R)\to\mathbf{HF}(R)$ be a homogeneous star operation on $R$, and let $\bigstar:\mathbf{F}(R)\to\mathbf{F}(R)$ be a (classical) star operation on $R$. Then we say that \emph{$\bigstar$ is an extension of $\star$}, in case $\star=\bigstar|_{\mathbf{HF}(R)}$. In other words, $\bigstar$ is an extension of $\star$, if for each element $A\in \mathbf{HF}(R)$, $A^{\bigstar}=A^{\star}$.
\end{definition}

The purpose of this note is to prove that every homogeneous star operation $\star:\mathbf{HF}(R)\to\mathbf{HF}(R)$ of $R$, has a largest extension $\e(\star)$ of $R$. That is we show that there exists a (classical) star operation $\e(\star)$ on $R$ such that $\star:=\e(\star)|_{\mathbf{HF}(R)}$, and that $\e(\star)$ is the largest (classical) star operation on $R$ with this property. More precisely in Section 2, we show that, for each $A\in \mathbf{F}(R)$,
$$
A^{\e({\star})}:=\bigcap\{z^{-1}(\mathbf{C}(zA))^{\star}\mid 0\neq
z\in (R:A)\},
$$
is a (classical) star operation on $R$ and it is an extension of $\star$ (Theorem \ref{e(*)}). We investigate the properties of $\e(\star)$; in particular we determine the set of homogeneous maximal $(\e(\star))_f$-ideals of $R$. In Section 3, we show that the sets $\HStar(R)$ and $\Star_{hp}(R)$ can be endowed with  Zariski-like topologies, such that the injection  $$\e:\HStar(R)\to\Star_{hp}(R), \star\mapsto\e(\star),$$ is a topological embedding (Proposition \ref{top}). Hence $\HStar(R)$ can be identifiable canonically with a subspace of $\Star_{hp}(R)$. We also show that the set $\HStar_f(R)$ of finite type homogeneous star operations on $R$ endowed with the sub-space topology of the Zariski topology of $\HStar(R)$ is a spectral space (Theorem \ref{spec}).

\subsection{Star operations}\footnote{We denote the concept of star operations described in \cite[Section 32]{g72} by (classical) star operations
to distinguish it clearly from the concept of homogeneous star operations introduced in \cite{k24}.}
Let $R$ be an integral domain with quotient
field $K$. Let $\mathbf{F}(R)$ be the set of
all nonzero \emph{fractional} ideals of $R$; i.e.,
$E\in\mathbf{F}(R)$ if $E$ is an $R$-submodule of $K$ and there
exists a nonzero element $r\in R$ with $rE\subseteq R$. Let $f(R)$
be the set of all nonzero finitely generated fractional ideals of
$R$. Obviously, $f(R)\subseteq\mathbf{F}(R)$.

Let $\star$ be a (classical) star operation on $R$. For every
$A\in \mathbf{F}(D)$, put $A^{\star_f}:=\bigcup B^{\star}$,
where the union is taken over all finitely generated ideals $B\in f(D)$
with $B\subseteq A$. It is easy to see that $\star_f$ is a (classical) star
operation on $R$, and ${\star_f}$ is called the \emph{star
operation of finite type associated to} $\star$. Note that
$(\star_f)_f=\star_f$. A star operation $\star$ is said to be of
\emph{finite type} if $\star=\star_f$; in particular ${\star_f}$ is
of finite type. It is known that if $\star$ is a homogeneous preserving star operation, then $\star_f$ is a homogeneous preserving star operation (see \cite[Lemma 2.4]{s14}).

It has become standard to say that a star operation $\star$ is
{\it stable} if $(A\cap B)^{\star}=A^{\star}\cap B^{\star}$ for all
$E$, $F\in \mathbf{F}(R)$. Given a
(classical) star operation $\star$ on $R$, it is possible to construct a
(classical) star operation $\widetilde{\star}$, which is stable and of
finite type defined as follows: for each
$A\in\mathbf{F}(R)$,
$$
A^{\widetilde{\star}}:=\{x\in K|xJ\subseteq A,\text{ for some
}J\subseteq R, J\in f(R), J^{\star}=R\},
$$
(see \cite[Section 2]{ac00}, where $\widetilde{\star}$ was denoted by $\star_w$). It is known that $\widetilde{\star}$ is a homogeneous preserving star operation, for each (classical) star operation $\star$ on $R$ \cite[Proposition 1]{s22}.

The most widely studied (classical) star operations on $R$ have been the
identity $d$, $v$, $t:=v_f$, and $w=\widetilde{v}$ operations,
where $A^{v}:=(A^{-1})^{-1}$, with $F^{-1}:=(R:F):=\{x\in
K|xF\subseteq R\}$ which are homogeneous preserving star operations (see \cite[Example 1]{s22}).

If $\star_1$ and $\star_2$ are (classical) star operations on $R$, one says
that $\star_1\leq\star_2$ if $A^{\star_1}\subseteq A^{\star_2}$ for
each $A\in \mathbf{F}(R)$. This
is equivalent to saying that $(A^{\star_1})^{\star_2}=A^{\star_2}=(A^{\star_2})^{\star_1}$ for
each $A\in \mathbf{F}(R)$ (cf. \cite[Page 1623]{aa90}).
Obviously, for each star operation $\star$ defined on $R$, we
have $\star_f\leq\star$. It is well-known that $d\leq\star\leq v$ for
all (classical) star operations $\star$ on $R$.

\subsection{Graded integral domains}
Let $\Gamma$ be a nonzero torsionless commutative cancellative monoid (written
additively), and $\langle \Gamma \rangle = \{a - b \mid a,b \in \Gamma\}$ be the
quotient group of $\Gamma$; so $\langle \Gamma \rangle$ is a torsionfree abelian group.
It is known that a cancellative monoid $\Gamma$ is torsionless
if and only if $\Gamma$ can be given a total order compatible with the monoid
operation \cite[page 123]{no68}. By a $(\Gamma$-)graded integral domain  $R =\bigoplus_{\alpha \in \Gamma}R_{\alpha}$,
we mean an integral domain graded by $\Gamma$.
That is, each nonzero $x \in R_{\alpha}$ has degree $\alpha$, i.e., deg$(x) = \alpha$,  and
deg$(0) = 0$. Thus, each nonzero $f \in R$ can be written uniquely as $f = x_{\alpha_1} + \dots + x_{\alpha_n}$ with
deg$(x_{\alpha_i}) = \alpha_i$ and $\alpha_1 < \cdots < \alpha_n$.
The simplest example of a
$\Gamma$-graded integral domain is the monoid domain
$D[\Gamma]$ of $\Gamma$ over
an integral domain $D$ with $\deg(aX^{\alpha})=\alpha$ for each
$0\neq a\in D$ and $\alpha\in \Gamma$.

An element $x \in R_{\alpha}$ for $\alpha \in \Gamma$ is said
to be {\em homogeneous}. Let $H$ be the saturated multiplicative set of nonzero homogeneous
elements of $R$. Then $H = \bigcup_{\alpha \in \Gamma}(R_{\alpha} \setminus \{0\})$, and
$R_H=\bigoplus_{\alpha\in\langle\Gamma\rangle}(R_H)_{\alpha}$ is a $\langle \Gamma \rangle$-graded integral domain whose nonzero homogeneous elements are units.
We say that $R_H$ is the {\em homogeneous quotient field} of $R$. For a
fractional ideal $I$ of $R$ let $I_h$ denote the fractional ideal
generated by the set of homogeneous elements of $R$ in $I$. An integral ideal $I$ of $R$ is said
to be \emph{homogeneous} if $I=\bigoplus_{\alpha\in\Gamma}(I\cap
R_{\alpha})$; equivalently, if $I=I_h$. A fractional ideal $I$ of
$R$ is \emph{homogeneous} if $sI$ is an integral homogeneous ideal
of $R$ for some $s\in H$ (thus $I\subseteq R_H$). For $f \in R_H$, let $\mathbf{C}(f)$ denote the
fractional ideal of $R$ generated by the homogeneous components of
$f$. For a fractional ideal $I$ of $R$ with $I \subseteq R_H$, let
$\mathbf{C}(I) = \sum_{f \in I} \mathbf{C}(f)$. Clearly, both $\mathbf{C}(f)$ and $\mathbf{C}(I)$ are
homogeneous fractional ideals of $R$ such that $f \in \mathbf{C}(f)$ and $I
\subseteq \mathbf{C}(I)$; and equality holds when $I$ is homogeneous. It is known that for $f, g\in R_H$ there exists an integer $m\geq2$ such that $\mathbf{C}(g)^m\mathbf{C}(f)=\mathbf{C}(g)^{m-1}\mathbf{C}(fg)$ by \cite{no68}. For more on graded integral domains and their divisibility
properties, see \cite{aa82,no68}.

Throughout this paper, let $\Gamma$ be a nonzero torsionless commutative cancellative monoid, $R=\bigoplus_{\alpha\in\Gamma}R_{\alpha}$ be an integral domain graded by $\Gamma$ with quotient field $K$, and $H$ be the set of nonzero homogeneous elements of $R$.

\section{Extension of homogeneous star operations}

Let $R=\bigoplus_{\alpha\in\Gamma}R_{\alpha}$ be a graded integral domain. The goal of the present section is to define in a canonical way, an extension to the graded domain $R$, of a given homogeneous star operation $\star:\mathbf{HF}(R)\to\mathbf{HF}(R)$ of $R$ (see Definition \ref{d}).

\begin{theorem}\label{e(*)}
Let $R =\bigoplus_{\alpha \in \Gamma}R_{\alpha}$ be a graded integral domain with quotient field $K$, and let $\star:\mathbf{HF}(R)\to \mathbf{HF}(R)$ be a homogeneous star operation of $R$. For each $A\in \mathbf{F}(R)$, set
$$
A^{\e({\star})}:=\bigcap\{z^{-1}(\mathbf{C}(zA))^{\star}\mid 0\neq z\in (R:A)\}.
$$
Then:
\begin{enumerate}
  \item $\e({\star})$ is a (classical) star operation on $R$.
  \item $\e({\star})$ is an extension of $\star$. In fact, it is the
  largest extension of $\star$.
  \item If $\star_1\leq\star_2$ are two homogeneous star operations
  in $R$, then $\e(\star_1)\leq\e(\star_2)$.
\end{enumerate}
\end{theorem}
\begin{proof}
It can be seen from the definition that $A^{\e({\star})}\in\mathbf{F}(R)$.

\textbf{Claim.} If $I\subseteq R$ is such that $I\cap H\neq\emptyset$, then $I^{\e({\star})}=(\mathbf{C}(I))^{\star}$.

Since $I\subseteq R$, we have $1\in(R:I)$. Hence $I^{\e({\star})}\subseteq(\mathbf{C}(I))^{\star}$. For the opposite
inclusion, let $z\in(R:I)\setminus(0)$. Since $I\cap H\neq\emptyset$, we obtain that $z\in R_H$. Assume that
$z=z_{\alpha_1}+\cdots+z_{\alpha_n}$ be the decomposition of $z$
into homogeneous elements in $R_H$. Then it can be seen that
$\mathbf{C}(zI)=z_{\alpha_1}\mathbf{C}(I)+\cdots+z_{\alpha_n}\mathbf{C}(I)$.
Therefore
\begin{align*}
z^{-1}(\mathbf{C}(zI))^{\star}=&z^{-1}(z_{\alpha_1}\mathbf{C}(I)+\cdots+z_{\alpha_n}\mathbf{C}(I))^{\star} \\[1ex]
        = & z^{-1}(z_{\alpha_1}(\mathbf{C}(I))^{\star}+
\cdots+z_{\alpha_n}(\mathbf{C}(I))^{\star})^{\star} \\[1ex]
        \supseteq&
z^{-1}(z_{\alpha_1}(\mathbf{C}(I))^{\star}+
\cdots+z_{\alpha_n}(\mathbf{C}(I))^{\star}) \\[1ex]
\supseteq& z^{-1}(z(\mathbf{C}(I)^{\star})=(\mathbf{C}(I))^{\star},
\end{align*}
where the second equality is from \cite[Proposition 2.4(1)]{k24}. Thus, we have $I^{\e({\star})}\supseteq(\mathbf{C}(I))^{\star}$ and
hence $I^{\e({\star})}=(\mathbf{C}(I))^{\star}$.

(1) It follows from the Claim that $R^{\e(\star)}=R$. For $x\in K\setminus\{0\}$, and $A\in \mathbf{F}(R)$, we now show that $(xA)^{\e({\star})}=xA^{\e({\star})}$. For this, note that $x^{-1}(R:A)=(R:xA)$ for each $x\in K\setminus
\{0\}$ by \cite[Lemma 2.1]{hhp98}. Hence
\begin{align*}
(xA)^{\e({\star})}=&\bigcap\{z^{-1}(\mathbf{C}(zxA))^{\star}\mid 0\neq z\in (R:xA)\} \\[1ex]
        = & \bigcap\{z^{-1}(\mathbf{C}(zxA))^{\star}\mid 0\neq z\in x^{-1}(R:A)\} \\[1ex]
        = & \bigcap\{xy^{-1}(\mathbf{C}(yA))^{\star}\mid 0\neq y\in (R:A)\}=xA^{\e({\star})}.
\end{align*}
Thus $(xA)^{\e({\star})}=xA^{\e({\star})}$. Now, from this, we obtain that $(x)^{\e({\star})}=(xR)^{\e({\star})}=xR^{\e({\star})}=xR=(x)$ for each $x\in K\setminus\{0\}$.

Let $z\in(R:A)\setminus(0)$ and $A\in\mathbf{F}(R)$. Then $zA\subseteq \mathbf{C}(zA)\subseteq (\mathbf{C}(zA))^{\star}$, and hence $A\subseteq z^{-1}(\mathbf{C}(zA))^{\star}$. Therefore $A\subseteq A^{\e({\star})}$.

It follows from the definition that if $A\subseteq B$ are fractional ideals of $R$, then $A^{\e({\star})}\subseteq B^{\e({\star})}$.

Now, we show that for each $A\in\mathbf{F}(R)$, $(A^{\e({\star})})^{\e({\star})}=A^{\e({\star})}$. From above paragraphs $A^{\e({\star})}\subseteq(A^{\e({\star})})^{\e({\star})}$. For the opposite inclusion, let $z\in(R:A)\setminus(0)$. By what we have proved so far, we have
$z(A^{\e({\star})})^{\e({\star})}=(zA^{\e({\star})})^{\e({\star})}=((zA)^{\e({\star})})^{\e({\star})}\subseteq((\mathbf{C}(zA))^{\star})^{\e({\star})}=
((\mathbf{C}(zA))^{\star})^{\star}=(\mathbf{C}(zA))^{\star}$, whence
$(A^{\e({\star})})^{\e({\star})}\subseteq z^{-1}(\mathbf{C}(zA))^{\star}$.
Therefore $(A^{\e({\star})})^{\e({\star})}\subseteq
A^{\e({\star})}$. So that $(A^{\e({\star})})^{\e({\star})}=A^{\e({\star})}$. Therefore $\e(\star)$ is a (classical) star operation on $R$.

(2) It follows from Claim that $\e(\star)$ extends $\star$. Indeed, let $A\in\mathbf{HF}(R)$, hence $A=\mathbf{C}(A)$. Then there exists $s\in H$ such that $sA\subseteq R$. Now from Claim we have $$sA^{\e({\star})}=(sA)^{\e({\star})}=(\mathbf{C}(sA))^{\star}=(s\mathbf{C}(A))^{\star}=sA^{\star},$$ and hence $A^{\e({\star})}=A^{\star}$.

In order to show that $\e(\star)$ is the largest star operation on $R$
that extends $\star$, assume that $\bigstar$ is another (classical) star
operation on $R$ that extends $\star$, that is for each $A\in\mathbf{HF}(R)$, $A^{\bigstar}=A^{\star}$. Let $A\in\mathbf{HF}(R)$, and $z\in(R:A)\setminus(0)$. Then $zA\subseteq \mathbf{C}(zA)$ and so
$zA^{\bigstar}=(zA)^{\bigstar}\subseteq(\mathbf{C}(zA))^{\bigstar}=(\mathbf{C}(zA))^{\star}$,
whence $A^{\bigstar}\subseteq z^{-1}(\mathbf{C}(zA))^{\star}$.
Hence $A^{\bigstar}\subseteq A^{\e(\star)}$, that is
$\bigstar\leq\e(\star)$.

(3) Follows from the definition.
\end{proof}

Let $*:\mathbf{HF}(R)\to\mathbf{HF}(R)$ be a homogeneous star operation on $R$, and for each $A\in\mathbf{HF}(R)$ let $A^{*_f}:=\bigcup\{B^* \mid B$ is a finitely generated homogeneous ideal with $B\subseteq A\},$ and $A^{\widetilde{*}}:=\{x\in R_H|xJ\subseteq A,$ for some finitely generated homogeneous ideal $J$, with $J^*=R\}$. Then $*_f$ and $\widetilde{*}$ are homogeneous star operation on $R$ \cite[Remark 2.9]{k24}. Also $*$ is called \emph{of finite type} of $*=*_f$.

For a given (classical) star operation $\star$ on $R$, we can always associate a stable (classical) star operation $\overline{\star}$ by defining, for every $F\in\mathbf{F}(R)$,
$$
F^{\overline{\star}}:=\{(F:I)\mid I \text{ is a nonzero ideal of }R\text{ such that }I^{\star}=R\}.
$$
It is easy to see that $\overline{\star}\leq\star$ and, moreover, that $\overline{\star}$ is the largest stable (classical) star operation that precedes $\star$. Therefore, $\star$ is stable if and only if $\star=\overline{\star}$, \cite[Proposition 3.7, Corollary 3.9]{fh00}, and \cite[Theorem 2.6]{ac00}.

In the following proposition we show that $\overline{\star}$ is always homogeneous preserving, for any (classical) star operation $\star:\mathbf{F}(R)\to\mathbf{F}(R)$ on a graded domain $R=\bigoplus_{\alpha\in\Gamma}R_{\alpha}$.

\begin{proposition}\label{} Let $R=\bigoplus_{\alpha\in\Gamma}R_{\alpha}$ be a graded integral domain, and $\star:\mathbf{F}(R)\to\mathbf{F}(R)$ be a (classical) star operation on $R$. Then, $\overline{\star}$ is homogeneous preserving.
\end{proposition}

\begin{proof} Let $I$ be a homogenous ideal of $R$. To show that $I^{\overline{\star}}$ is homogeneous let $f\in I^{\overline{\star}}$. Then $fJ\subseteq I$ for some ideal $J$ of $R$ such that $J^{\star}=R$. Thus $R=J^{\star}\subseteq \mathbf{C}(J)^{\star}\subseteq R$, and hence $\mathbf{C}(J)^{\star}=R$, and that $\mathbf{C}(f)\mathbf{C}(J)\subseteq I$. Thus, $\mathbf{C}(f)\subseteq I^{\overline{\star}}$. This means that $I^{\overline{\star}}$ is homogeneous. Now assume that $A$ is a homogeneous fractional ideal of $R$. Then there exists $0\neq a\in H$ such that $I:=aA\subseteq R$. Hence $aA^{\overline{\star}}=I^{\overline{\star}}$ is homogeneous. Therefore $A^{\overline{\star}}=a^{-1}(aA^{\overline{\star}})$ is a homogeneous fractional ideal.
\end{proof}

We can use of the extension $\e(*)$ of a homogeneous star operation $*:\mathbf{HF}(R)\to \mathbf{HF}(R)$ to extend various properties of (classical)star operations to homogeneous star operations. In particular \cite[Proposition 2.4]{k24} can follows easily from this extension and \cite[Proposition 32.2]{g72}. See also the proof of Lemma \ref{dd} for another use of this extension. Also we can use of this construction to associate new homogeneous star operations, to a given homogeneous star operation. For example it can be seen that $*_f=\e(*)_f|_{\mathbf{HF}(R)}$ and $\widetilde{*}=\widetilde{\e(*)}|_{\mathbf{HF}(R)}$. For another new construction of homogeneous star operations see the following remark that associates a homogeneous star operation $\overline{*}$ to a given homogeneous star operation $*:\mathbf{HF}(R)\to \mathbf{HF}(R)$.

\begin{remark}\label{stable}{\em
Let $R =\bigoplus_{\alpha \in \Gamma}R_{\alpha}$ be a graded integral domain, let $\star:\mathbf{HF}(R)\to\mathbf{HF}(R)$ be a homogeneous star operation of $R$. Then the homogeneous star operation $\overline{\star}:=\overline{\e(\star)}|_{\mathbf{HF}(R)}:\mathbf{HF}(R)\to\mathbf{HF}(R)$ is stable, that is for any $A, B\in\mathbf{HF}(D)$, we have $(A\cap B)^{\overline{\star}}=A^{\overline{\star}}\cap B^{\overline{\star}}$. In fact for every $A\in\mathbf{HF}(D)$,
$$
A^{\overline{\star}}=\{(A:I)\mid I \text{ is a nonzero homogeneous ideal of }D\text{ such that }I^{\star}=D\}.
$$
That the right hand side in a subset of $A^{\overline{\star}}$ is clear. For the opposite inclusion let $a\in A^{\overline{\star}}$ be a homogeneous element. Then there exists an ideal $I\subseteq R$ such that $I^{\e(\star)}=R$ and $aI\subseteq A$. Hence $R=I^{\e(\star)}\subseteq\mathbf{C}(I)^{\star}\subseteq R$; so $\mathbf{C}(I)^{\star}=R$, and that $a\mathbf{C}(I)\subseteq A$. Thus $a$ belongs to the right hand side.}
\end{remark}

\begin{theorem}\label{2.4}
Let $R =\bigoplus_{\alpha \in \Gamma}R_{\alpha}$ be a graded integral domain, let $\star:\mathbf{HF}(R)\to\mathbf{HF}(R)$ be a homogeneous star operation of $R$, and let $\e(\star)$ be the extension of $\star$ to $R$ introduced in Theorem \ref{e(*)}. Then:
\begin{enumerate}
  \item $(\e(\star))_f=(\e(\star_f))_f$ is the largest finite type extension of $\star_f$.
  \item $\widetilde{\e(\star)}=\widetilde{\e(\widetilde{\star})}$ is the largest stable and finite type extension of $\widetilde{\star}$.
  \item $\overline{\e(\star)}=\overline{\e(\overline{\star})}$ is the largest stable extension of $\overline{\star}$.
\end{enumerate}
\end{theorem}
\begin{proof}
(1) Since $\e(\star)$ is an extension of $\star$, it follows that
$(\e(\star))_f$ is an extension of $\star_f$. Indeed, let $A\in
\mathbf{HF}(R)$ and $a\in A^{\star_f}$. Then there exists a
finitely generated homogeneous fractional ideal $B\subseteq A$ such
that $a\in B^{\star}=B^{\e(\star)}\subseteq A^{\e(\star)}$. Hence
$A^{\star_f}\subseteq A^{(\e(\star))_f}$. For the opposite inclusion
let $a\in A^{(\e(\star))_f}$. Then there exists a finitely generated
fractional ideal $B\subseteq A$ such that $a\in
B^{\e(\star)}\subseteq A^{\star}$. Assume that $B=(f_1,\ldots,f_n)$. Then $B\subseteq B':=\sum_{i=1}^n\mathbf{C}(f)$, and $B'$ is a finitely generated homogeneous fractional ideal of $R$ such that $B'\subseteq A$  and that $a\in(B')^{\e(\star)}=(B')^{\star}$. Hence $a\in A^{\star_f}$. Then $A^{\star_f}\supseteq A^{(\e(\star))_f}$, i.e.,
$A^{\star_f}=A^{(\e(\star))_f}$. Therefore $(\e(\star))_f$ extends
$\star_f$. By Theorem \ref{e(*)}(2), $(\e(\star))_f\leq\e(\star_f)$ and
hence $(\e(\star))_f\leq(\e(\star_f))_f$. Since the opposite
inequality is obvious, we have $(\e(\star))_f=(\e(\star_f))_f$. Now,
let $\bigstar$ be a finite type extension of $\star_f$, then
$\bigstar\leq\e(\star_f)$ by Theorem \ref{e(*)}, and hence
$\bigstar=\bigstar_f\leq(\e(\star_f))_f=(\e(\star))_f$.

(2) We will show first that $\widetilde{\e(\star)}$ is an extension
of $\widetilde{\star}$. Let $A\in \mathbf{HF}(R)$ and $f\in
A^{\widetilde{e(\star)}}$. Then $fJ\subseteq A$ for some finitely
generated ideal $J\subseteq R$ such that $J^{\e(\star)}=R$. Suppose that
$J=(g_1,\cdots,g_n)$. Using \cite{hhp98}, there is an
integer $m\geq2$ such that $\mathbf{C}(g_i)^m\mathbf{C}(f)=\mathbf{C}(g_i)^{m-1}\mathbf{C}(fg_i)$ for
all $i=1,\cdots,n$. Since $A$ is homogeneous and $fg_i\in A$, we
have $\mathbf{C}(fg_i)\subseteq A$. Thus we have $\mathbf{C}(g_i)^m\mathbf{C}(f)\subseteq A$.
Since $J\subseteq(\mathbf{C}(g_1)+\cdots+\mathbf{C}(g_n))$ we have
$(\mathbf{C}(g_1)+\cdots+\mathbf{C}(g_n))^{\star}=(\mathbf{C}(g_1)+\cdots+\mathbf{C}(g_n))^{\e(\star)}=R$.
Put $J_0:=(\mathbf{C}(g_1)+\cdots+\mathbf{C}(g_n))^{nm}$. Thus $J_0$ is a finitely
generated homogeneous ideal of $R$ such that $J^{\star}_0=R$ and
that $\mathbf{C}(f)J_0\subseteq A$, hence $\mathbf{C}(f)\subseteq
A^{\widetilde{\star}}$. This means that
$A^{\widetilde{\e(\star)}}\subseteq A^{\widetilde{\star}}$. For the
opposite inclusion assume that $f\in A^{\widetilde{\star}}$. Then
$fJ\subseteq A$ for some finitely generated homogeneous ideal $J$ of $R$ such
that $J^{\star}=R$. But $J^{\e(\star)}=J^{\star}$. Thus
$A^{\widetilde{\e(\star)}}\supseteq A^{\widetilde{\star}}$, and
hence $A^{\widetilde{\e(\star)}}=A^{\widetilde{\star}}$. Therefore
$\widetilde{\e(\star)}$ is an extension of $\widetilde{\star}$. By Theorem \ref{e(*)}(2),
$\widetilde{\e(\star)}\leq\e(\widetilde{\star})$ and hence
$\widetilde{\e(\star)}\leq\widetilde{\e(\widetilde{\star})}$. Since
the opposite inclusion is obvious, we have
$\widetilde{\e(\star)}=\widetilde{\e(\widetilde{\star})}$. Let
$\bigstar$ be a stable finite-type (classical) star operation on $R$
that extends $\widetilde{\star}$. Then
$\bigstar\leq\e(\widetilde{\star})$ and hence
$\bigstar=\widetilde{\bigstar}\leq\widetilde{\e(\widetilde{\star})}=
\e(\widetilde{\star})$.

(3) Since $\e(\star)$ is an extension of $\star$, it follows that $\overline{\e(\star)}$ is an extension of $\overline{\star}$ by Remark \ref{stable}.
Then by Theorem \ref{e(*)}(2), $\overline{\e(\star)}\leq\e(\overline{\star})$ and hence $\overline{\e(\star)}\leq\overline{\e(\overline{\star})}$. Since the opposite
inequality is obvious, we have $\overline{\e(\star)}=\overline{\e(\overline{\star})}$. Now, let $\bigstar$ be a stable extension of $\overline{\star}$, then $\bigstar\leq\e(\overline{\star})$ by Theorem \ref{e(*)}, and hence $\bigstar=\overline{\bigstar}\leq\overline{\e(\overline{\star})}=\overline{\e(\star)}$.
\end{proof}

In the following corollary, for a homogeneous preserving star operation $\star:\mathbf{F}(R)\to\mathbf{F}(R)$, we still denote by $\star$, the homogeneous star operation $\star|_{\mathbf{HF}(R)}$.

\begin{corollary}\label{2.3}
Let $R =\bigoplus_{\alpha \in \Gamma}R_{\alpha}$ be a graded
integral domain, and let $v$, $t$ and $w$ be the $v$-star
operation, $t$-star operation and $w$-star operation on $R$
respectively. Then:
\begin{enumerate}
  \item $\e(v)=v$.
  \item $(\e(v))_f=(\e(t))_f=t$.
  \item $\widetilde{\e(v)}=\widetilde{\e(w)}=w$.
\end{enumerate}
\end{corollary}
\begin{proof}
(1) Observe that $v$ is the largest (classical) star operation that is also an extension of the homogeneous star operation $v=v|_{\mathbf{HF}(R)}:\mathbf{HF}(R)\to\mathbf{HF}(R)$. Hence by Theorem \ref{e(*)} we obtain that $\e(v)=v$. Now (2) and (3) are clear from (1) and
Theorem \ref{2.4}.
\end{proof}

Recall that a nonzero fractional ideal $I$ of $R$ is called \emph{$v$-invertible} if $(II^{-1})^{v}=R$ and a domain $R$ is a \emph{Pr\"{u}fer $v$-multiplication domain (for short, P$v$MD)} if every finitely generated ideal of $R$ is $v$-invertible. Pr\"{u}fer $v$-multiplication domains generalize at the same time Pr\"{u}fer and and Krull domains (\cite{gr67} and \cite{mz81}). It is well-known that a graded domain $R =\bigoplus_{\alpha \in \Gamma}R_{\alpha}$ is a P$v$MD if and only if $I$ is $v$-invertible for each homogeneous ideal of $R$ \cite[Theorem 6.4]{aa82}. Recall that a graded domain $R$ is called a \emph{graded-Pr\"{u}fer domain} if each nonzero finitely generated homogeneous ideal of $R$ is invertible (see \cite{ac13}). Hence, it can be seen that every graded-Pr\"{u}fer domain is a P$v$MD. Since every P$v$MD is integrally closed, we see that every graded-Pr\"{u}fer domain is integrally closed.

Let $D$ be an integral domain and $\{D_{\alpha}\}$ be a family of overrings of $D$ such that $D=\cap_{\alpha}D_{\alpha}$. Then the map $*:\mathbf{F}(D)\to\mathbf{F}(D)$ defined by $F^*:=\cap_{\alpha}FD_{\alpha}$, for each $F\in\mathbf{F}(D)$ is a (classical) star operation \cite[Theorem 32.5]{g72}. A particular case is the $b:=b_D$-operation defined by $F^b:=\cap_{\alpha}FV_{\alpha}$, for each $F\in\mathbf{F}(D)$, where $\{V_{\alpha}\}$ is the class of all valuation overrings of $D$, see \cite[Page 398]{g72}. The $b$-operation is a (classical) star operation, only in the case that $D$ is integrally closed, but it is always a semistar operation. It is shown that the $b$-operation is a homogeneous preserving semistar operation \cite[Corollary 2]{s22}; and in particular, if $R =\bigoplus_{\alpha \in \Gamma}R_{\alpha}$ is an integrally closed graded domain, it is a homogeneous preserving star operation, when restricted to fractional ideals.

\begin{remark}\label{}{\em
(1) Let $R =\bigoplus_{\alpha \in \Gamma}R_{\alpha}$ be a graded
integral domain. Assume that $\star_1$ and $\star_2$ are two homogeneous preserving star operations on $R$ such that $\star_1|_{\mathbf{HF}(R)}=\star_2|_{\mathbf{HF}(R)}$. Then $\e(\star_1)=\e(\star_2)$.

(2) Note that $d\neq(\e(d))_f$ and $b\neq(\e(b))_f$. Let $R$ be a graded-Pr\"{u}fer domain which is not Pr\"{u}fer (e.g. $R=D[X,X^{-1}]$ for a Pr\"{u}fer domain $D$ which is not a field  and an indeterminate $X$ over $D$ \cite[Example 3.6]{ac13}). Since $R$ is integrally closed and is not a Pr\"{u}fer domain, there exists an ideal $I$ of $R$ such that $I^{t}\neq I$ by \cite[Proposition 34.12]{g72}. Hence $t\neq d$. However $d|_{\mathbf{HF}(R)}=t|_{\mathbf{HF}(R)}$ since $R$ is a graded-Pr\"{u}fer domain \cite[Corollary 3.4]{s18}. Therefore $\e(d)=\e(t)$ and hence $(\e(d))_f=(\e(t))_f=t\neq d$. For the second case, we note that $b|_{\mathbf{HF}(R)}=d|_{\mathbf{HF}(R)}$ by \cite[Corollary 3.8]{s18}. Assume to the contrary that $b=(\e(b))_f$. Then $b=(\e(b))_f=(\e(d))_f=t$. Since $R$ is a graded-Pr\"{u}fer domain, we have $R$ is a P$v$MD. Therefore $R$ is a Pr\"{u}fer domain by \cite[Proposition 9]{fp11}, which is a contradiction.
}
\end{remark}

Let $\star:\mathbf{F}(R)\to\mathbf{F}(R)$ (resp. $\star:\mathbf{HF}(R)\to\mathbf{HF}(R)$) be a (classical) star operation (resp. homogeneous star operation) on $R$. An ideal (resp. homogeneous ideal) $I$ of $R$ is called a \emph{$\star$-ideal} if $I^{\star}=I$. A $\star$-ideal (resp homogeneous $\star$-ideal) of $R$ is called a {\em maximal $\star$-ideal} (resp. {\em homogeneous maximal $\star$-ideal}) if it is maximal among proper (resp. proper homogeneous) $\star$-ideals of $R$. Denote by $\Max^{\star}(R)$ (resp. $h\text{-}\Max^{\star}(R)$) the set of maximal (resp. homogeneous maximal) $\star$-ideals of $R$. It is easy to see that each proper homogeneous $\star$-ideal of $R$ is contained in a homogeneous maximal $\star$-ideal of $R$. Also, it is easy to see that a homogeneous maximal ideal need not be a maximal ideal, while a homogeneous maximal $t$-ideal is a maximal $t$-ideal \cite[Lemma 1.2]{ac05}. Since we have $(\e(t))_f=t$ by Corollary \ref{2.3}, we can consider the following proposition as a generalization of \cite[Lemma 1.2]{ac05}.

\begin{proposition}\label{max}
Let $R =\bigoplus_{\alpha \in \Gamma}R_{\alpha}$ be a graded integral domain and $\star:\mathbf{HF}(R)\to\mathbf{HF}(R)$ be a homogeneous star operation on $R$. Then each maximal $(\e(\star_f))_f$-ideal of $R$ intersecting $H$ is homogeneous.
\end{proposition}
\begin{proof}
Let $M$ be a maximal $(\e(\star_f))_f$-ideal of $R$ intersecting $H$. If $M\subsetneq\mathbf{C}(M)$, then we have
$\mathbf{C}(M)^{\star_f}=R$; so there are $f_1,\ldots,f_n\in M$ such that $\mathbf{C}(f_1,\ldots,f_n)^{\star_f}=R$. Let $a\in M\cap H$. Then for each $z\in(R:(a,f_1,\ldots,f_n))\setminus(0)\subseteq R_H$, we obtain that  $\mathbf{C}(z(a,f_1,\ldots,f_n))^{\star_f}=(\mathbf{C}(za)+\mathbf{C}(z(f_1,\ldots,f_n)))^{\star_f}
=((\mathbf{C}(za))^{\star_f}+\mathbf{C}(z(f_1,\ldots,f_n))^{\star_f})^{\star_f}=(a\mathbf{C}(z)^{\star_f}+\mathbf{C}(z)^{\star_f})^{\star_f}=
\mathbf{C}(z)^{\star_f}$. Hence
\begin{align*}
M\supseteq &(a,f_1,\ldots,f_n)^{\e({\star_f})} \\[1ex]
        =&\bigcap\{z^{-1}(\mathbf{C}(z(a,f_1,\cdots,f_n)))^{\star_f}\mid 0\neq z\in (R:(a,f_1,\ldots,f_n))\} \\[1ex]
        = & \bigcap\{z^{-1}(\mathbf{C}(z))^{\star_f}\mid 0\neq z\in (R:(a,f_1,\ldots,f_n))\} \\[1ex]
        \supseteq & \bigcap\{z^{-1}(zR)^{\e(\star_f)}\mid 0\neq z\in (R:(a,f_1,\ldots,f_n))\}\\[1ex]
        =&R^{\e(\star_f)}=R.
\end{align*}
This is a contradiction. Therefore, $M=\mathbf{C}(M)$.
\end{proof}

We state Proposition \ref{max} and the following corollaries for homogeneous star operations and it is clear that these are true for homogeneous preserving star operations too.

\begin{corollary}\label{}
Let $R =\bigoplus_{\alpha \in \Gamma}R_{\alpha}$ be a graded integral domain and $\star:\mathbf{HF}(R)\to\mathbf{HF}(R)$ be a homogeneous star operation on $R$. Then
$$h\text{-}\Max^{\star_f}(R)=\{M\in \Max^{(\e(\star_f))_f}(R)\mid M\cap H\neq\emptyset\}.$$
\end{corollary}

\begin{corollary}\label{}
Let $R =\bigoplus_{\alpha \in \Gamma}R_{\alpha}$ be a graded integral domain. Then:
\begin{enumerate}
  \item $h\text{-}\Max^{t}(R)=\{M\in \Max^{t}(R)\mid M\cap H\neq\emptyset\}.$
  \item $h\text{-}\Max(R):=h\text{-}\Max^{d_R}(R)=\{M\in \Max^{(\e(d_R))_f}(R)\mid M\cap H\neq\emptyset\}.$
\end{enumerate}
\end{corollary}

It is well-known that $\Max^{\star_f}(R)=\Max^{\widetilde{\star}}(R)$ for each (classical) star operation $\star:\mathbf{F}(R)\to\mathbf{F}(R)$ of $R$, see \cite[Theorem 2.16]{ac00}. Then we reobtain the following equality, see \cite[Proposition 2.5]{s14} and \cite[Corollry 2.11]{k24}.

\begin{corollary}\label{}
Let $R =\bigoplus_{\alpha \in \Gamma}R_{\alpha}$ be a graded integral domain and $\star:\mathbf{HF}(R)\to\mathbf{HF}(R)$ be a homogeneous star operation on $R$. Then
$$h\text{-}\Max^{\star_f}(R)=h\text{-}\Max^{\widetilde{\star}}(R).$$
\end{corollary}

\section{Topological considerations}

Let $D$ be an integral domain with quotient
field $K$. Let $\overline{\mathbf{F}}(D)$ denotes the set of all nonzero $D$-submodules of $K$. Obviously,
$\mathbf{F}(D)\subseteq\overline{\mathbf{F}}(D)$. As in \cite{om94}, a {\it semistar operation on} $D$ is a map
$\star:\overline{\mathbf{F}}(D)\rightarrow\overline{\mathbf{F}}(D)$, $E\mapsto E^{\star}$, such that, for all $x\in K$, $x\neq 0$, and
for all $E, F\in\overline{\mathbf{F}}(D)$, the properties (1) $(xE)^{\star}=xE^{\star}$; (2) $E\subseteq F$ implies that $E^{\star}\subseteq F^{\star}$; (3) $E\subseteq E^{\star}$ and $E^{\star\star}:=(E^{\star})^{\star}=E^{\star}$ hold. If $\star:\overline{\mathbf{F}}(D)\rightarrow\overline{\mathbf{F}}(D)$ is a semistar operation on $D$ such that $D^{\star}=D$, then $\star|_{\mathbf{F}(D)}:\mathbf{F}(D)\rightarrow\mathbf{F}(D)$ is a (classical) star operation. On the other hand if, $\star:\mathbf{F}(D)\rightarrow\mathbf{F}(D)$ is a (classical) star operation, then $\star_e:\overline{\mathbf{F}}(D)\rightarrow\overline{\mathbf{F}}(D)$ defined by
$$F^{\star_e}=\left\{
  \begin{array}{ll}
    F^{\star} & \hbox{$F\in\mathbf{F}(D);$} \\
    K, & \hbox{$F\in\overline{\mathbf{F}}(D)\setminus\mathbf{F}(D),$}
  \end{array}
\right.$$
for each $F\in\overline{\mathbf{F}}(D)$, is a semistar operation such that $\star_e|_{\mathbf{F}(D)}=\star$, see \cite[Remark 1.5]{fh00}.

In \cite{fs14}, the set $\SStar(D)$ of all semistar operation on $D$, was endowed with a topology (called the \emph{Zariski topology}) having, as a subbasis of open sets, the sets of the type $V_E:=\{\star\in\SStar(D)\mid1\in E^{\star}\}$, where $E$ is a nonzero $D$-submodule of $K$. This topology makes $\SStar(D)$ into a quasi-compact $T_0$ space.

On the set $\Star(D)$ of all (classical) star operations on $D$, we can introduce a natural topology that we still call the \emph{Zariski topology}, whose subbasic  open sets are $$U_F:=\{\star\in\Star(D)\mid 1\in F^{\star}\},$$ as $F$ varies among the elements of $\mathbf{F}(D)$.

\begin{remark}\label{}
{\em
Endow $\Star(D)$ and $\SStar(D)$ with their Zariski topologies. Then the map $$e:\Star(D)\to\SStar(D), \star\mapsto\star_e,$$ is a topological embedding. Indeed, if $V_E$ is a subbasic open set of $\SStar(D)$, then $e^{-1}(V_E)=\{\star\in\Star(D)\mid1\in E^{\star_e}\}$. So that $e^{-1}(V_E)=U_E$ if $E\in\mathbf{F}(D)$, and $e^{-1}(V_E)=\Star(D)$ if $E\in\overline{\mathbf{F}}(D)\setminus\mathbf{F}(D)$. Thus $e$ is continuous. On the other hand $e(U_F)=V_F\cap e(\Star(D))$ for each $F\in\mathbf{F}(D)$. Therefore $e$ is a topological embedding.
}
\end{remark}

Let $R =\bigoplus_{\alpha \in \Gamma}R_{\alpha}$ be a graded integral domain. Consider the set $\Star_{hp}(R)$ of all homogeneous preserving star operations as a topological space endowed with the subspace topology of the Zariski topology of $\Star(R)$.

On the set $\HStar(R)$ of all homogeneous star operations on $R$, we introduce a natural topology that we still call the \emph{Zariski topology}, whose subbasic  open sets are $$W_A:=\{\star\in\HStar(R)\mid 1\in A^{\star}\},$$ as $A$ varies among the elements of $\mathbf{HF}(R)$. Consider the set $\HStar_f(R)$ of all finite type homogeneous star operations as a topological space endowed with the subspace topology of the Zariski topology of $\HStar(R)$.

\begin{remark}\label{f}
{\em Let $R=\bigoplus_{\alpha\in\Gamma}R_{\alpha}$ be a graded integral domain.
The Zariski topology on $\HStar_{f}(R)$ is determined by the finitely generated homogeneous fractional ideals of $R$, in the sense that the collection of the sets of the form $W'_A:=W_A\cap\HStar_{f}(R)$, where $A$ varies among the finitely generated homogenous fractional ideals of $R$, is a subbasis, since there is an equality $$W_F\cap\HStar_{f}(R)=\bigcup\{W'_A\mid A\in f_h(R), A\subseteq F\},$$ where $f_h(R)$ is the set of finitely generated homogenous fractional ideals.
}
\end{remark}

In the following proposition, the closure of a subset $Y$ of a topological space is denoted by $\Ad(Y)$.

\begin{proposition}\label{T0}
Let $R =\bigoplus_{\alpha \in \Gamma}R_{\alpha}$ be a graded integral domain. Endow $\HStar(R)$ and $\Star_{hp}(R)$ with their Zariski topologies.
\begin{enumerate}
  \item For any $*\in\HStar(R)$, we have $\Ad(\{*\})=\{*'\in\HStar(R)\mid*'\leq*\}$.
  \item The topological space $\HStar(R)$ is $T_0$.
  \item The canonical map $\varphi:\Star_{hp}(R)\to\HStar(R)$, $\star\mapsto \star|_{\mathbf{HF}(R)}$, is a continuous surjection.
  \item $\varphi\circ\e$ is the identity map of $\HStar(R)$, that is, $\varphi$ is a topological retraction.
\end{enumerate}
\end{proposition}
\begin{proof}
(1) For $``\supseteq"$ assume $*'\in\HStar(R)$ such that $*'\leq*$ and that $A\in\mathbf{HF}(R)$ is so that $*'\in W_A$. Then $1\in A^{*'}\subseteq A^*$, and hence $*\in W_A$. This means that $*'\in\Ad(\{*\})$. For $``\subseteq"$, if $*'\nleq*$, then there is $A\in\mathbf{HF}(R)$ such that $A^{*'}\nsubseteq A^{*}$; hence, if $a\in A^{*'}\setminus A^{*}$ is a homogeneous element, then $*\notin W_{a^{-1}A}$ while $*'\in W_{a^{-1}A}$. Thus $\HStar(R)\setminus W_{a^{-1}A}$ is a closed set containing $*$ but not $*'$, and $*'\notin\Ad(\{*\})$.

(2) By (1), for any homogeneous star operations $*,*'$ on $R$, $\Ad(\{*\})=\Ad(\{*'\})$ if and only if $*=*'$, i.e., $\HStar(R)$ satisfies the $T_0$ axiom.

It follows from Theorem \ref{e(*)} that $\varphi$ is surjective. Since $$\varphi^{-1}(W_A)=U_A\cap\Star_{hp}(R),$$ for each $A\in\mathbf{HF}(R)$, $\varphi$ is continuous, to complete the proof of (3) and (4).
\end{proof}

In the following proposition we show that $\HStar(R)$ can be identifiable canonically with a subspace of $\Star_{hp}(R)$ endowed with a weaker topology generated by the family $\{U_F\cap\Star_{hp}(R)\mid F\in\mathbf{F}(R), F\subseteq R_H\}$. %So that we called the above topology on $\HStar(R)$, the Zariski topology.

\begin{proposition}\label{top}
Let $R =\bigoplus_{\alpha \in \Gamma}R_{\alpha}$ be a graded
integral domain. Endow $\HStar(R)$  with its Zariski topology. If $\Star_{hp}(R)$ is endowed with the topology generated by the family $\{U_F\cap\Star_{hp}(R)\mid F\in\mathbf{F}(R), F\subseteq R_H\}$, then the map $$\e:\HStar(R)\to\Star_{hp}(R), \star\mapsto\e(\star),$$ is a topological embedding.
\end{proposition}
\begin{proof}
First, we show that $\e$ is continuous, and it is enough to show that, for any fractional ideal $F$ of $R$ such that $F\subseteq R_H$, the set $\e^{-1}(U_F\cap\Star_{hp}(R))$ is open in $\HStar(R)$. Note that $$\e^{-1}(U_F\cap\Star_{hp}(R))=\{*\in\HStar(R)\mid 1\in F^{\e(*)}\}.$$ Fix a homogeneous star operation $*\in\e^{-1}(U_F\cap\Star_{hp}(R))$. Let $0\neq z\in(R:F)$. Then $zF\subseteq R$ and hence $zF^{\e(*)}\subseteq R$. Since $1\in F^{\e(*)}$, we obtain that $z\in R$. Because $1\in F^{\e(*)}\subseteq \mathbf{C}(F)^{\e(*)}=\mathbf{C}(F)^{*}$ we have $*\in W_T$, where $T:=\mathbf{C}(F)$. Now we show that $W_T\subseteq \e^{-1}(U_F\cap\Star_{hp}(R))$. Assume that $\star\in W_T$. Then $1\in T^{\star}=\mathbf{C}(F)^{\star}$. Then for each $0\neq z\in (R:F)$ we have $\mathbf{C}(zF)^{\star}=(z_{\alpha_1}\mathbf{C}(F)+\cdots+z_{\alpha_n}\mathbf{C}(F))^{\star}=
(z_{\alpha_1}\mathbf{C}(F)^{\star}+\cdots+z_{\alpha_n}\mathbf{C}(F)^{\star})^{\star}$, where $z=z_{\alpha_1}+\cdots+z_{\alpha_n}$ is the decomposition of $z$ into homogeneous elements in $R$. Then
\begin{align*}
F^{\e(\star)}=& \bigcap\{z^{-1}(\mathbf{C}(zF))^{\star}\mid 0\neq z\in(R:F)\} \\[1ex]
             =& \bigcap\{z^{-1}(z_{\alpha_1}\mathbf{C}(F)^{\star}+\cdots+z_{\alpha_n}\mathbf{C}(F)^{\star})^{\star}\mid 0\neq z\in(R:F)\} \\[1ex]
             \supseteq& \bigcap\{z^{-1}(z_{\alpha_1}\mathbf{C}(F)^{\star}+\cdots+z_{\alpha_n}\mathbf{C}(F)^{\star})\mid 0\neq z\in(R:F)\} \\[1ex]
             \supseteq& \mathbf{C}(F)^{\star},
\end{align*}
and hence $1\in \mathbf{C}(F)^{\star}\subseteq F^{\e(\star)}$ and $\e(\star)\in U_F\cap\Star_{hp}(R)$. Therefore $\e^{-1}(U_F\cap\Star_{hp}(R))$ is an open set in $\HStar(R)$.

Now, we have to show that the image via $\e$ of each open set $W$ of $\HStar(R)$, is open in $\e(\HStar(R))$ (endowed with the subspace topology). Without loss of generality, we can assume that $W=W_A$ for some $A\in \mathbf{HF}(R)$. Then we have $\e(W_A)=\e(\HStar(R))\cap U_A$. This shows that $\e(W_A)$ is open in $\e(\HStar(R))$.
\end{proof}

Let $\mathcal{S}$ be a nonempty set of homogeneous star operations on $R$. For each $A\in\mathbf{HF}(R)$, define $\bigwedge(\mathcal{S})$ as follows:
$$A^{\bigwedge(\mathcal{S})}=\bigcap\{A^*\mid*\in\mathcal{S}\}.$$
It is easy to see that $\bigwedge(\mathcal{S})$ is a homogeneous star operation on $R$ and it is the infimum of $\mathcal{S}$ in the partially ordered set $(\HStar(R),\leq)$. The homogeneous star operation $\bigvee(\mathcal{S}):=\bigwedge\{\tau\mid *\leq\tau,$ for all $*\in\mathcal{S}\}$ is the supremum of $\mathcal{S}$ in $(\HStar(R),\leq)$. %In fact by the same resoning $(\HStar(R),\leq)$ is a complete lattice.

\begin{proposition}\label{2.7} Let $\mathcal{S}$ be a quasi-compact subspace of $\HStar_f(R)$. Then, the homogeneous star operation $\bigwedge(\mathcal{S})$ belongs to $\HStar_f(R)$.
\end{proposition}
\begin{proof}
Set $\mathcal{S}=\{*_i\mid i\in I\}$, $*:=\bigwedge(\mathcal{S})$, fix $A\in\mathbf{HF}(R)$ and let $a\in A^*$ be homogeneous. Since $A^*=\bigcap_{i\in I}A^{*_i}$, and each $*_i$ is of finite type, there are finitely generated homogeneous ideals $B_i\subseteq A$ such that $a\in B_i^{*_i}$; thus, for any $i$, $1\in a^{-1}B_i^{*_i}=(a^{-1}B_i)^{*_i}$ and $*_i\in W_{a^{-1}B_i}=:\Omega_i$. Therefore, $\{\Omega_i\mid i\in I\}$ is an open cover of $\mathcal{S}$. By compactness there is a finite subcover $\{\Omega_{i_1},\ldots,\Omega_{i_n}\}$. Set $B:=B_{i_1}+\cdots+B_{i_n}\subseteq A$; we claim that $a\in B^*$, and this implies that $*$ is of finite type.

For each $i\in I$, there is a $\Omega_{i_j}$ such that $*_i\in\Omega_{i_j}$. Hence $*_i\in W_{a^{-1}B_{i_j}}$ and $1\in(a^{-1}B_{i_j})^{*_i}$, that is $a\in B_{i_j}^{*_i}\subseteq B^{*_i}$. Therefore, $a\in\bigcap_{i\in I}B^{*_i}=B^*$.
\end{proof}

\begin{proposition}\label{2.11} Let $\{W_{A_i}\mid i\in I\}$ be a nonempty family of subbasic open sets of the Zariski topology of $\HStar(R)$. The following statements hold.
\begin{itemize}
  \item[(1)] $\bigcap\{W_{A_i}\mid i\in I\}$ is a complete lattice (as a subset of the partially ordered set $(\HStar(R),\leq)$.
  \item[(2)] $\bigcap\{W_{A_i}\mid i\in I\}$ is a quasi-compact subspace of $\HStar(R)$. In particular, $W_A$ is quasi-compact for any $A\in \mathbf{HF}(R)$.
\end{itemize}
\end{proposition}
\begin{proof}
Set $W:=\bigcap_{i\in I}W_{A_i}$ and let $\Delta$ be a nonempty subset of $W$. Then $\mathfrak{s}:=\bigvee(\Delta)$ and $\mathfrak{i}:=\bigwedge(\Delta)$ are, respectively, the supremum and the infimum of $\Delta$ in $\HStar(R)$. Thus it suffices to show that $\mathfrak{s},\mathfrak{i}\in W$. Clearly, $\mathfrak{s}\in W$, because $\mathfrak{s}\geq*$, for any $*\in\Delta$, and thus $\mathfrak{s}$ belongs to each open set containing some $*$. Furthermore, for any $i\in I$ we have $$1\in\bigcap_{*\in W_{A_i}}A_i^{*}\subseteq\bigcap_{*\in W}A_i^{*}=:A_i^{\mathfrak{i}},$$ and thus $\mathfrak{i}\in W$. Hence (1) is proved.

Now let $\mathcal{U}$ be an open cover of $W$. By (1), $\mathfrak{i}\in W$, and thus there is an open set $U_0\in \mathcal{U}$ such that $\mathfrak{i}\in U_0$. Then, $U_0$ must contain the whole $W$. Statement (2) is now clear.
\end{proof}

\begin{lemma}\label{dd} Let $\emptyset\neq Y\subseteq\HStar_f(R)$. Then $\bigvee(Y)\in\HStar_f(R)$ and, for any $A\in \mathbf{HF}(R)$, we have
$$A^{\bigvee(Y)}=\bigcup\{A^{\sigma_1\circ\cdots\circ\sigma_n}\mid\sigma_1,\ldots,\sigma_n\in Y,n\in \mathbb{N}\}.$$
\end{lemma}
\begin{proof} Let $Y':=\{(\e(\star))_f\mid\star\in Y\}$ which is a subset of $\Star_f(R)$. It follows from \cite[Page 1628]{aa90} that, $\bigvee(Y')$ is a (classical) star operation of finite type, and that for any $F\in \mathbf{F}(R)$, $F^{\bigvee(Y')}=\bigcup\{F^{\sigma'_1\circ\cdots\circ\sigma'_n}\mid\sigma'_1,\ldots,\sigma'_n\in Y',n\in \mathbb{N}\}$. Note that $\bigvee(Y)=\bigvee(Y')|_{\mathbf{HF}(R)}$. Hence $\bigvee(Y)\in\HStar_f(R)$ and $$A^{\bigvee(Y)}=\bigcup\{A^{\sigma_1\circ\cdots\circ\sigma_n}\mid\sigma_1,\ldots,\sigma_n\in Y,n\in \mathbb{N}\},$$ for any $A\in \mathbf{HF}(R)$.
\end{proof}

A topological space $X$ is called a \emph{spectral} space if $X$ is quasi-compact %(i.e. every open cover has a finite subcover)
and $T_0$, the quasi-compact open subsets of $X$ are closed under finite intersection and form an open basis, and every nonempty irreducible closed subset of $X$ has a generic point (i.e. it is the closure of a unique point). By a theorem of Hochster, these are precisely the topological spaces which arise as the prime spectrum of a commutative ring endowed with the Zariski topology \cite{h69}.

Now we show $\HStar_{f}(R)$ is a spectral space, whose proof is analogous to \cite[Thorem 2.13]{fs14}.

\begin{theorem}\label{spec} The space $\HStar_{f}(R)$ of finite type homogeneous star operations on $R$, endowed with the Zariski topology, is a spectral space.
\end{theorem}

\begin{proof} In order to prove that a topological space $X$ is a spectral space, we use the characterization given in \cite[Corollary 3.3]{f14}. We recall that if $\beta$ is a non empty family of subsets of $X$, for a given subset $Y$ of $X$ and an ultrafilter $\mathcal{U}$ on $Y$, we set $$Y_{\beta}(\mathcal{U}):=\{x\in X\mid[\text{for each }B\in\beta\text{, it happens that }x\in B\Leftrightarrow B\cap Y\in \mathcal{U}]\}.$$
By \cite[Corollary 3.3]{f14}, for a topological space $X$ being a spectral space is equivalent to $X$ being a $T_0$-space having a subbasis $\mathcal{S}$ for the open sets such that $X_{\mathcal{S}}(\mathcal{U})\neq\emptyset$, for every ultrafilter $\mathcal{U}$ on $X$.

We know $X:=\HStar_{f}(R)$ is a $T_0$-space by Proposition \ref{T0}(2), and set $\mathcal{S}:=\{W'_A\mid A\in f_h(R)\}$ be the canonical subbasis of the Zariski topology on $X$. Let $\mathcal{U}$ be an ultrafilter on $X$. It suffices to show the set $$X_{\mathcal{S}}(\mathcal{U}):=\{\star\in X\mid[\text{for each }W'_A\in\mathcal{S}\text{, it happens that } \star\in W'_A\Leftrightarrow W'_A\in\mathcal{U}]\}$$ is nonempty. By Propositions \ref{2.7} and \ref{2.11}, any homogeneous star operation of the form $\bigwedge(W'_A)$ (where $A\in f_h(R)$) belongs to $\HStar_f(R)$. Thus the homogeneous star operation $$\star:=\bigvee(\{\bigwedge(W'_A)\mid W'_A\in\mathcal{U}\})\in\HStar_f(R)=X.$$ We claim that $\star\in X_{\mathcal{S}}(\mathcal{U})$. Fix a finitely generated homogeneous fractional ideal $A$ of $R$. It suffices to show $\star\in W'_A\Leftrightarrow  W'_A\in\mathcal{U}$. First, assume $\star\in W'_A$, i.e., $1\in A^{\star}$. By Lemma \ref{dd}, there exist finitely generated fractional ideals $A_1,\ldots,A_n$ of $R$ such that $1\in A^{\bigwedge(W'_{A_1})\circ\cdots\circ\bigwedge(W'_{A_n})}$ and $W'_{A_i}\in \mathcal{U}$, for any $i=1,\ldots,n$. Take a homogeneous star operation $\sigma\in\cap_{i=1}^nW'_{A_i}$. By definition, $\sigma\geq\bigwedge(W_{A_i})$, for $i=1,\ldots,n$, and thus $$1\in A^{\bigwedge(W'_{A_1})\circ\cdots\circ\bigwedge(W'_{A_n})}\subseteq A^{\sigma\circ\cdots\circ\sigma}=A^{\sigma},$$ thus $\sigma\in W'_A$. This shows that $\cap_{i=1}^nW'_{A_i}\subseteq W'_A$ and thus, by definition of ultrafilter, $W'_A\in\mathcal{U}$, since $W'_{A_1},\ldots,W'_{A_n}\in\mathcal{U}$. Conversely, assume that $W'_A\in\mathcal{U}$. This implies that $\bigwedge(W'_A)\leq\star$. By definition, $1\in A^{\sigma}$, for each $\sigma\in W'_A$, and thus $$1\in\bigcap_{\sigma\in W'_A}A^{\sigma}=:A^{\bigwedge(W'_A)}\subseteq A^{\star}.$$ This completes the proof.
\end{proof}

\begin{corollary}
Let $D$ be an integral domain. The space $\Star_{f}(D)$ of finite type (classical) star operations on $D$, endowed with the Zariski topology, is a spectral space.
\end{corollary}
\begin{proof} The domain $D$ is trivially a $\Gamma$-graded domain with $\Gamma=\{0\}$. Therefore in this case $\HStar_f(D)=\Star_f(D)$, and hence $\Star_f(D)$ is a spectral space by Theorem \ref{spec}.
\end{proof}

\begin{center} {\bf ACKNOWLEDGMENT}

\end{center}
I would like to thank the onymous referee for his/her insightful report on this paper. On behalf of all authors, the corresponding author states that there is no conflict of interest.

\end{document}